\documentclass{amsart}
\usepackage{amssymb,amsmath,amsxtra, xcolor}
\newtheorem{theorem}{Theorem}
\newtheorem{corollary}[theorem]{Corollary}

\newtheorem{proposition}[theorem]{Proposition}

\newtheorem{application}[theorem]{Application}
\newtheorem{lemma}[theorem]{Lemma}
\newtheorem{remark}[theorem]{Remark}

\newcommand{\C}{{\mathbb C}}

\newcommand{\N}{{\mathbb{N}}}
\newcommand{\R}{{\mathbb{R}}}

\newcommand{\s}{\,^*\!}
\newcommand{\st}{\,^\circ}
\newcommand{\sr}{\,^*{\mathbb R}}

\newcommand{\abs}[1]{\left\vert#1\right\vert}
\newcommand{\norm}[1]{\left\Vert#1\right\Vert}

\newcommand{\dis}[1]{\displaystyle#1}

\begin{document}

\baselineskip=18pt

\begin{center}
{\textbf{BIDUAL AS A WEAK NONSTANDARD HULL}}
\end{center}

\begin{center}

Siu-Ah Ng\footnote{
Address: School of Mathematical Sciences, University of KwaZulu-Natal,
Pietermaritzburg, 3209 South~Africa\\ website:
\texttt{http://www.maths.unp.ac.za/{\~{ }}siuahn/default.htm}
\quad \quad email: ngs@ukzn.ac.za }

\end{center}

\bigskip

\centerline{Abstract}

\begin{quote}
\small {\em We construct the weak nonstandard hull of a normed linear space $X$ from $\s X$ (the nonstandard extension of $X$) using the weak topology on $X.$ The bidual (i.e. the second dual) $X^{\prime\prime}$ is shown to be isometrically isomorphic to the weak nonstandard hull of $X.$ Examples and applications to C*-algebras are given, including a simple proof of the Sherman-Takeda Theorem. As a consequence, the weak nonstandard hull of a C*-algebra is always a von Neumann algebra. Moreover a natural representation of the Arens product is given. }
\end{quote}

\bigskip

\noindent Every normed linear space $X$ extends naturally to the bidual $X^{\prime\prime}.\,$ Less well-known among functional analysts is that $X$ also has a natural extension to a nonstandard version $\s X$ using methods from the Nonstandard Analysis. In a sense, any extension of $X$ with respect to some formal properties can always be identified from $\s X.$ The aim of this article is to relate $X^{\prime\prime}$ and $\s X$ and exploit the link between them. The main tool comes from a modification of W.A.J. Luxemburg's nonstandard hull construction (\cite{A}).

In \S 1 we give a very brief summary of the methodology and terminologies from nonstandard analysis. A generalization of the nonstandard hull construction is given. In \S 2 a representation of the bidual of a standard normed linear space $X$ as the weak nonstandard hull of $X$ is constructed from $\s X.$ In \S 3 applications of this representation to some sequence spaces are given. In \S 4 we apply our results to C*-algebras. Based on the weak nonstandard hull representation of the bidual, we produce a simple nonstandard proof of the Sherman-Takeda Theorem that the bidual of a C*-algebra forms a von Neumann algebra. In particular, this shows that the weak nonstandard hull of a C*-algebra is always a von Neumann algebra. Moreover a natural representation is provided for the Arens product(s) on the bidual.

\bigskip

\section{Preliminaries and the general nonstandard hull construction}

Background from nonstandard analysis is summarized as follows.

The nonstandard extension of a standard mathematical object $\,X\,$ is denoted by $\,\s X.\,$ (Note the various usages of the star symbol in this article.) The extensions are done simultaneously for all ordinary mathematical objects under consideration and with the preservation of all set theoretical properties among the extensions expressible in the first order logic in the language consisting of the membership symbol. This is referred to as the \emph{Transfer Principle}. In particular we have extensions such as $\s\N,\,$  $\sr\,$  and $\s\C$ which behave with respect to each other in the same formal manner as $\N,$ $\R$ and $\C.$ Since we regard $X\subset \s X,$ any element $a$ in $X$ is also written as $\s a,$ depending on the emphasis. An element from some $\,\s X\,$ is referred to as an \emph{internal} set. Since $X\in\mathcal{P} (X),$ the power set, each $\,\s X\,$ itself is internal; but there are internal sets not of this form. Non-internal sets are called \emph{external}. We identify a property with the set it defines, so we may speak of $\s P$ when $P$ is a standard mathematical property.

Elements in the set $\s\N$ are called \emph{hyperfinite}; a set counted internally by a hyperfinite number is also called hyperfinite (this is the same as $\s$finite); given $r, s\in\s\R,\,$ if $\abs{r-s}<1/n$ for all $n\in\N,$ we write $r\approx s$ (\emph{infinitely close}); $r$ is called \emph{infinitesimal} when $r\approx 0;$ a finite element $r$ of $\s\R$ (written $\abs{r}<\infty$) is one with $\abs{r}<n$ for some $n\in\N$; such $r$ is $\approx s$ for a unique $s\in\R$ (called the \emph{standard part}; in symbol: $s=\st r$). We use similar notions for elements in $\,\s\C.\,$

For some uncountable cardinal $\kappa$ sufficiently large for our purpose, we assume throughout that the so-called \emph{$\kappa$-Saturation Principle} is satisfied in our universe of nonstandard objects (which is possible under a weaker form of the Axiom of Choice), namely:
\begin{quote} If \emph{$\,\mathcal{F}$ is a family of no more than $\kappa$ internal sets such that $\displaystyle{\,\bigcap{\mathcal F}_0\neq\emptyset}$ for any finite subfamily $\mathcal{F}_0$ of $\,\mathcal{F}\,$ (i.e. $\mathcal{F}$ satisfies the finite intersection property),} then $\displaystyle{\,\bigcap{\mathcal F}\neq\emptyset}.$
\end{quote}

We refer the readers to \cite{A} for details of the construction of nonstandard universe and the methodology of nonstandard analysis.

Alternatively, material in this article can be formulated in the less intuitive and more complicated language of ultraproducts, namely one regards $\s X$ as some ultrapower $\prod_{\text{U}}X$ and internal subsets of $\s X$ as the some ultraproduct $\prod_{\text{U}}X_i$ for some $X_i\subset X$ and a strong enough but fixed ultrafilter $\text{U}.$

From a standard normed linear space, the nonstandard hull construction, due to Luxemburg (\cite{A}), produces a Banach space extension in the standard sense. Here we describe a generalization of this method.

Let $X$ be an internal linear space over $\s \C.$ Let $\rm W$ be a (possibly external) set of internal seminorms on $X.\,$ So for each $p\in {\rm W},\,$ $\,p:X\to\s [0,\infty)\,$ and
\[\forall x,y\in X\,\forall \alpha\in\s\C\; \big[p(x+y)\leq p(x)+p(y)\,\land\, p(\alpha x)=\abs{\alpha}p(x)\big].\]

Write $\text{Fin}(X)$ for $\{x\in X\,\vert\, \sup_{p\in {\rm W}} \st p(x)<\infty \},$ the finite part of $X$ w.r.t. $\rm W.$ On $\text{Fin}( X)$ an equivalence relation $\approx_{\rm{w}}$ is defined:  $x_1\approx_{\rm{w}} x_2$ iff $\,\forall p\in{\rm W}\, [p(x_1)\approx p(x_2)].$ Let $\mathfrak{I}:=\{x\in \text{Fin}(X)\,\vert\,x\approx_{\rm{w}} 0\}.$ Noticing that $\text{Fin}(X)$ is a linear space in the standard sense and $\mathfrak{I}$ is a subspace in the standard sense, i.e.  closed under addition and multiplication by $\alpha\in\C\,$ (in fact even by finite $\alpha\in\s\C$), we can form the quotient space $\text{Fin}(X)/\mathfrak{I},\,$ and denote it by $\widehat{X}^{\rm{w}}.\,$ Elements $x+\mathfrak{I}$  of $\widehat{X}^{\rm{w}}\,$ are denoted by $\widehat{x},\,$ where $x\in \text{Fin}(X).\,$ For $\widehat{x},\widehat{y}\in \widehat{X}^{\rm{w}}$ and $\alpha\in \C,\; \widehat{x}+\alpha \widehat{y}\,$ is defined as $\widehat{x+\beta y},$ for any $\beta\approx\alpha.$ Moreover,
\[\norm{\widehat{x}}_{\rm{w}}:=\sup\big\{\,\st p(x)\;\vert\; p\in{\rm W}\,\big\}.\]
It is straightforward to check that all these are well-defined.

\begin{proposition}
$\norm{\cdot }_{\rm{w}}$ forms a norm on $\widehat{X}^{\rm{w}}$ under which $\widehat{X}^{\rm{w}}$ is a standard Banach space.
\end{proposition}

\begin{proof}
It is easy to see that $\widehat{X}^{\rm{w}}$ is a standard linear space  normed by $\norm{\cdot }_{\rm{w}}.$

Completeness follows from the $\kappa$-saturation with $\kappa$ chosen to be $\geq (\omega_1+\abs{{\rm W}})^+\,$ as follows.

Let $\{ \widehat{x}_n\,\vert\,n\in\N\}\,$ be a Cauchy sequence in $\widehat{X}^{\rm{w}}.$ For $m\in\N$ let $k_m\in\N\,$ so that
\[\forall n\in\N\, \big[ n>k_m \Rightarrow \norm{\widehat{x}_n-\widehat{x}_{k_m}}_{\rm{w}}<\frac{1}{2m}\,\big].\]
By $\omega_1$-saturation, we can extend $\{ x_n\,\vert\,n\in\N\}\,$ to some internal hyperfinite sequence $\{ x_n\,\vert\,n<N\}\,$ in $X,\,$ for some $N\in\s\N.\,$ Let $\mathcal{F}=\{\mathcal{F}_{p,m}\,\vert\, p\in{\rm W},\, m\in \N\,\},\,$ where $\mathcal{F}_{p,m}:= \{ x_n\,\vert\, n<N\;\land\; p({x_n-x_{k_m}})\leq 1/2m\,\}.$ Clearly, $\mathcal{F}$ is a family of internal sets having the finite intersection property. Therefore, by $\kappa$-saturation, we can find some $x\in\bigcap\mathcal{F}.\,$ From the definitions, we have
\[\forall n, m\in\N\;\big[ n>k_m\Rightarrow  \sup_{p\in {\rm W}} \st p(x-x_n)\leq\frac{1}{m}\,\big]\]
i.e. $\;x\in \text{Fin}(X)\,$ and $\displaystyle{\lim_{n\to\infty} \norm{\widehat{x}-\widehat{x}_n }_{\rm{w}} =0. }$

Hence $\widehat{X}^{\rm{w}}$ is a Banach space under $\norm{\cdot }_{\rm{w}}.$
\end{proof}

\bigskip

When ${\rm W}=\{\,\norm{\cdot } \},$ i.e. an internal norm $\norm{\cdot}$ on $X,$ this construction coincides with Luxemburg's nonstandard hull.

\bigskip

\section{The weak nonstandard hull and the bidual}

From now on $X$ always stands for a standard normed linear space over $\C.\,$ The dual space is denoted by $X^\prime$ and hence the bidual (second dual) by $X^{\prime\prime}.$ Each bounded linear form $\phi\in X^\prime$ defines an internal seminorm on $\s X:$
\[p_\phi (x) \,:=\,\abs{\s\phi (x)},\quad\text{where} \; x\in\s X. \]

We further restrict ourselves to the case ${\rm W}=\{\, p_\phi\,\vert\, \phi\in X^\prime,\,\norm{\phi}\leq 1\,\}\,$ and write $\widehat{X},\, \approx$  and $\,\norm{\cdot}\,$ instead of $\widehat{\s X}^{\rm{w}},\,\approx_{\rm w}\,$ and $\,\norm{\cdot}_{\rm{w}}.\,$

We call $\widehat{X}$ \emph{the weak nonstandard hull} of $X.$

Note that $\text{Fin}(\s X)$ includes $\{x\in\s X\,\vert\, \norm{x}<\infty \}.\,$ It is generally a proper subset. Moreover, on $\text{Fin}(\s X),$  $x_1\approx x_2$ iff $\,\forall \phi\in X^\prime\,\big[ \s\phi (x_1)\approx \s\phi (x_2)\big].$

Also $\mathfrak{I}=\{x\in \text{Fin}(\s X)\,\vert\,\s\phi (x)\approx 0,\,\phi\in X^\prime\}.$

We identify $X\subset \widehat{X}\,$ as a subspace via the isomorphic embedding $x\mapsto \widehat{x}.\,$

The following result identifies the bidual with the weak nonstandard hull of $X.$

\begin{theorem}\label{bidual}
The Banach space $\widehat{X}$ is isometrically isomorphic to $X^{\prime\prime}.$
\end{theorem}

\begin{proof}
We first define $\pi : \widehat{X}\to X^{\prime\prime}.\,$ For $\widehat{x}\in \widehat{X}$ and $\phi\in X^\prime,$ we let
$\pi (\widehat{x})\big(\phi\big):= \st\big(\s\phi(x)\big).$ So $\pi (\widehat{x})$ is a well-defined bounded linear form on $X^\prime.$

Clearly $\pi$ is injective and a homomorphism. We now show that it is surjective and an isometry.

Under enough saturation, let $A$ be a hyperfinite set such that $\,X^\prime \subset A\subset \s\big( X^\prime \big).$ Here the inclusion $X^\prime\subset \s\big(X^\prime\big)$ is given by identifying each $\phi\in X^\prime$ with $\s\phi.$ Let $x^{\prime\prime}\in X^{\prime\prime}$ and $0<\epsilon\approx 0.$ By transferring Helley's Theorem (\cite{M}~1.9.12), there is $a\in \s X$ such that
\[\norm{a}\leq \norm {\s x^{\prime\prime}} +\epsilon\quad\text{and}\quad \forall\,\phi\in A\,\big[\phi(a) =\s x^{\prime\prime} (\phi)\big].\]
In particular, $a\in \text{Fin}(\s X)$ and $\forall\,\phi\in X^\prime \,\big[\s\phi(a) \approx \big(x^{\prime\prime}\big) (\phi)\big].$

Consequently, $\pi (\widehat{a}) = x^{\prime\prime}$ and $\norm{\widehat{a}}= \norm{x^{\prime\prime}}.$

Therefore $\pi$ is surjective and an isometry.
\end{proof}

\bigskip

We remark that only $\kappa$-saturation for an uncountable $\kappa$ greater than the algebraic dimension of $X^\prime$ is needed in the above proof.

From now on, $\widehat{X}$ is identified with $X^{\prime\prime}.$

\bigskip

\begin{corollary}\label{wc}
$X$ is a reflexive Banach space iff $X=\widehat{X}\,$ iff the closed unit ball of $X$ is weakly compact.
\end{corollary}

\begin{proof}
By the above theorem and the canonical embedding of $X$ into $\widehat{X},\,$ reflexivity is equivalent to $X=\widehat{X},\,$ which by definition is equivalent to every point $a$ in the unit ball of $\s X$ is infinitely close to a standard point $c\in X\,$ in the weak topology, i.e. $\s \phi(a)\approx \phi (c)\,$ for all $\phi\in X^\prime.\,$ So by Robinson's characterization of compactness (\cite{A}), it is equivalent to the weak compactness of the close unit ball of $X.$
\end{proof}

\bigskip

A similar application leads to a proof of Goldstine's Theorem: The closed unit ball of $X$ is weak* dense in the closed unit ball of $X^{\prime\prime}.$

The following gives an alternative way of computing the norm in $\widehat{X}$ from the norm in $\s X.$

\bigskip

\begin{theorem}\label{norm}
Let $\widehat{a}\in \widehat{X}.$ Then $\dis \norm{\widehat{a}} = \inf\big\{ \st\norm{x}\,\big\vert\, x\in{\rm{Fin}}(\s X),\, x\approx a\,\big\}.$

In other words, $\dis \norm{\widehat{a}} = \inf_{\varepsilon\in\mathfrak{I}} \st\norm{a+\varepsilon}.$
\end{theorem}

\begin{proof}
For convenience, we write $\norm{\widehat{a}}_{\rm{v}} = \inf_{\varepsilon\in\mathfrak{I}} \st\norm{a+\varepsilon}.$

We let $B_X$ and $\bar{B}_{X^\prime}$ denote the open ball $\{x\in X\,\vert\, \norm{x}<1\}$ and the closed ball $\{\phi\in {X^\prime}\,\vert\, \norm{\phi}\leq 1\}$ respectively.

Let $\varepsilon\in\mathfrak{I},$ then
\[\norm{\widehat{a}}\,=\,\sup_{\phi\in \bar{B}_{X^\prime}}\st\abs{\s\phi(a)}\,=\,\sup_{\phi\in \bar{B}_{X^\prime}}\st\abs{\s\phi(a+\varepsilon )}\,\leq\,\st\norm{a+\varepsilon},\]
hence $\norm{\widehat{a}}\leq \norm{\widehat{a}}_{\rm{v}}.$

To show $\norm{\widehat{a}}_{\rm{v}}\leq \norm{\widehat{a}}$ we first assume without loss of generality that $\norm{\widehat{a}} = 1,$ and will show that $\norm{\widehat{a}}_{\rm{v}} \leq 1.$

\bigskip

\noindent \emph{Claim}: Let $A\subset \{\phi\in X^\prime\,\vert\,\norm{\phi}=1\,\}$ be finite and $n,m\in\N.$ Then there exists some $c\in (1+\frac{1}{m}) B_X$ such that $\forall \phi\in A\, \forall n\in N\;\st\abs{\s\phi(a)-\phi(c)}\leq \frac{1}{n}.$

We denote $\st(\s\phi(a))$ by $r_\phi.$ Suppose the claim fails. Then for some finite $A\subset \{\phi\in X^\prime\,\vert\,\norm{\phi}=1\,\}$ and $n,m\in\N,$ we have
\[\forall x\in X\, \Big[\bigwedge_{\phi\in A}\Big(\abs{\phi(x)-r_\phi}\leq\frac{1}{n}\Big)\,\Rightarrow\,x\notin \big(1+\frac{1}{m}\big) B_X\,\Big].\]
i.e. the set $\big\{ x\in X\,\vert \bigwedge_{\phi\in A}\abs{\phi(x)-r_\phi}\leq\frac{1}{n}\big\}\,$ is disjoint from $(1+\frac{1}{m}) B_X.$ Moreover, both sets are convex and the latter is open. So by the Hahn-Banach Separation Theorem (\cite{R}~3.4), for some $\theta\in X^\prime$ and $\ell\geq 0,$
\begin{align*}
\big(1+\frac{1}{m}\big) B_X\,&\subset\, \{x\in X\,\vert\, {\rm Re}(\theta(x))<\ell\}\\
\big\{ x\in X\,\vert \bigwedge_{\phi\in A}\abs{\phi(x)-r_\phi}\leq\frac{1}{n}\big\}\,&\subset\,\{x\in X\,\vert\, {\rm Re}(\theta(x))\geq\ell\}.
\end{align*}
That is, for any $x,y\in X,\,$ whenever $\norm{x}<1$ and $\bigwedge_{\phi\in A}\abs{\phi(y)-r_\phi}\leq\frac{1}{n},$ then
\begin{equation}\label{e1}
{\rm Re}(\theta (x))\,<\,\frac{\ell}{1+1/m}\,<\,\ell\,\leq\, {\rm Re}(\theta (y)).
\end{equation}

By scaling, we can assume that $\norm{\theta}=1.$

By saturation, for some $x\in\s B_X,$ $\norm{\theta}\approx \abs{\s\theta(x)}.$  Replace such $x$ by $x e^{-i{\rm Arg}(\theta (x))},$ we can assume that $\norm{\theta}\approx \s\theta(x)\in\s\R.$ Then by transferring (\ref{e1}),
\[1\,=\,\norm{\theta}\approx \s\theta(x)\, < \,\frac{\ell}{1+1/m}\,<\,\ell\,\leq\, {\rm Re}(\theta (a))\,\leq\,\abs{\theta(a)}\]
which gives $\norm{\widehat{a}} > 1,$ a contradiction. Hence the Claim is proved.

From the Claim, by saturation and the transfer, we have
\[\exists c\in\s X\, \big[\st\norm{c}\leq 1\, \land \,\forall \phi\in X^\prime\,\big( \s\phi(a)\approx \s\phi(c)\big)\big],\]
i.e. $c\approx a$ and $\st\norm{c}\leq 1.$ Therefore we have $\norm{\widehat{a}}_{\rm{v}}\leq 1$ as desired.
\end{proof}

\bigskip

\begin{corollary}\label{ca}
Let $\widehat{a}\in \widehat{X}.$ Then $\dis \norm{\widehat{a}} \approx \norm{a+\varepsilon}$ for some $\varepsilon\in\mathfrak{I}.$
\end{corollary}

\begin{proof}
By the above theorem, for each $n\in \N$ there is  $\varepsilon_n\in\mathfrak{I}$ so that
\[\Big\vert\norm{\widehat{a}}\,-\,\norm{a+\varepsilon_n}\Big\vert\,\leq\,\frac{1}{n}.\]
Extend $\{\varepsilon_n\}$ to an internal sequence, let $\varepsilon =\varepsilon_n$ for any $n\in\s\N\setminus\N\,$ and notice for each $n\in\s\N$ that $\dis \norm{\varepsilon_n}\leq 2\norm{a}+\frac{1}{n} <\infty.$
\end{proof}

\bigskip

\section{The spaces $c_0,\, \ell_1,\, \ell_\infty\,$ and $ba(\N)$}

In this section, we draw from results in the previous section some interesting conclusions about certain Banach spaces of complex sequences. We let $c_0,\, \ell_1,\, \ell_\infty\,$ respectively denote the Banach space of complex sequences which converge to $0,$ which are summable and which are uniformly bounded. $c_0$ and $\ell_\infty$ are both given the supremum norm and each $a\in \ell_1$ is given the norm $\sum_{n\in\N}\abs{a_n}.\,$ We let $ba(\N)\,$ to denote the Banach space of finitely additive complex measures defined for all subsets of $\N,\,$ with the total variation norm. (See \cite{D}.)

It is well-known that $c_0^\prime = \ell_1,\,$ $\ell_1^\prime =\ell_\infty\,$ and $\ell_\infty^\prime = ba(\N).$

\bigskip

\begin{application}
Let $X=c_0.$ Then
\[{\rm Fin}(\s X)\,=\,{\rm Fin}(\s c_0)\,=\,\big\{a\in\s c_0\,\vert\, \forall b\in \ell_1\,\abs{\sum_{n\in\s\N} a_n\s b_n}<\infty\, \big\}\]
and, in it, $a\approx c$ iff $\dis \forall b\in \ell_1\,\sum_{n\in\s\N} a_n\s b_n\approx \sum_{n\in\s\N} c_n\s b_n.\,$

Let $\pi : \widehat{X}\to X^{\prime\prime}$ be given by Theorem~\ref{bidual}, i.e. $\pi : \widehat{c_0}\to \ell_\infty.\,$

Let $\widehat{a}\in\widehat{c_0}\,$ and let $c\in\ell_\infty\,$ denote $\pi(\widehat{a}).\,$ Then from the definitions and transfer, we have for all $b\in\ell_1\,$ that
\[\sum_{n\in\s\N} a_n\s b_n =\s b (a)\approx \pi (\widehat{a}) (b)= c(b)=\sum_{n\in\N}c_n b_n = \sum_{n\in\s \N}\s c_n \s b_n. \]
By taking $b={\rm 1}_{\{ n\}},\,n\in \N,\,$ we have $a_n\approx c_n\,$ for all $n\in\N.$

Therefore, $\pi(\widehat{a}) = (\st a_n)_{n\in \N}.\,$ In particular, for $a\in {\rm Fin}(\s c_0),$
\[\forall b\in\ell_1\,\sum_{n\in\s N}a_n\s b_n\approx 0\quad\text{iff}\quad \forall n\in\N\; a_n\approx 0.\]

Hence, given any $a\in\s c_0\,$ with $a_n\approx 0$ for all $n\in \N,\,$ if $\dis \sum_{n\in\s \N}a_n\s b_n\not\approx 0\,$ for some $b\in\ell_1,\,$ then $\dis \sum_{n\in\s \N}a_n\s b_n\,$ is infinite for some other $b\in\ell_1.\,$

Moreover, as a consequence of Theorem~\ref{norm}, given any $a\in {\rm Fin}(\s c_0),\,$ there exists $c\in {\rm Fin}(\s c_0),\,$ such that
\[\forall b\in\ell_1\,\sum_{n\in\s \N}a_n\s b_n\approx\sum_{n\in\s \N}c_n\s b_n\quad\text{and}\quad \st\sup_{n\in\N}\abs{c_n}=\sup_{b\in\ell_1,\,\norm{b}=1}\st \sum_{n\in\s \N}a_n\s b_n. \]
\hfill $\Box$
\end{application}

\bigskip

A similar application of Theorem~\ref{bidual} and Theorem~\ref{norm} gives the following concrete representation of measures in $ba(\N).\,$

\bigskip

\begin{application}
Let $X=\ell_1.$ Then
\[{\rm Fin}(\s X)\,=\,{\rm Fin}(\s \ell_1)\,=\,\big\{a\in\s \ell_1\,\vert\, \forall b\in \ell_\infty\,\abs{\sum_{n\in\s\N} a_n\s b_n}<\infty\, \big\}\]
and here $a\approx c$ iff $\dis \forall b\in \ell_\infty\,\sum_{n\in\s\N} a_n\s b_n\approx \sum_{n\in\s\N} c_n\s b_n.\,$

Let $\pi : \widehat{X}\to X^{\prime\prime}$ be given by Theorem~\ref{bidual}, i.e. $\pi : \widehat{\ell_1}\to ba(\N).\,$

Therefore for each $\mu\in ba(\N),\,$ there is $a\in{\rm Fin}(\s \ell_1)\,$ such that $\pi(\widehat{a})=\mu.\,$

Let $S\subset \N,\,$ so ${\rm 1}_S\in\ell_\infty\,$ and
\[\mu (S) = \int_S d\mu = \mu ({\rm 1}_S)=\pi(\widehat{a})({\rm 1}_S)\approx \sum_{n\in\s\N} a_n \s{\rm 1}_S(n). \]

That is, by the above and Theorem~\ref{norm}, for any $\mu\in ba(\N),\,$ there is $a\in\s \ell_1\,$
\[\forall S\in \N\; \mu (S) =\st \sum_{n\in\s S} a_n\quad\text{and}\quad \norm{\mu} =\st \sum_{n\in\s\N}\abs{a_n}. \]
\hfill $\Box$
\end{application}

\bigskip

\section{The bidual of a C*-algebra}

In this section we will show that the bidual of a C*-algebra forms a von Neumann algebra.

Following the tradition, the involution of an element $x$ in an C*-algebra is denoted by $x^*,\,$ not to be confused with the nonstandard extension $\s x.$

Throughout this section $X$ is taken to be a standard C*-algebra. Recall that $X^{\prime\prime}$ is identified with $\widehat{X},$ hence has elements of the form $\widehat{a},\,$ where $ a\in \s X.$

Let $H$ be the Hilbert space corresponding to the universal representation of $X$ (\cite{B}~III.5.2.1). Let $\mathcal{B}(H)$ denote the C*-algebra of bounded linear operators on $H$ and identify $X$ as a C*-subalgebra of $\mathcal{B}(H)$ under this representation.

For $\xi,\rho\in H,$ we let $\omega_{\xi,\rho}$ denote the linear form in $\big(\mathcal{B}(H)\big)^\prime$ that takes $x\in \mathcal{B}(H)$ to $\langle x(\xi ),\rho\rangle.$ We also regard $\omega_{\xi,\rho}\in X^\prime\,$ when dealing only with its restriction on $X.\,$ The definition extends internally to $\omega_{\xi,\rho}$ for $\xi,\rho\in \s H.$ It is clear that $\norm{\omega_{\xi,\rho}}\leq \norm{\xi}\,\norm{\rho}.$ In particular, $\,\omega_{\xi,\rho}\in X^\prime\,$ for $\,\xi,\rho\in H.\,$

\bigskip

\begin{proposition}\label{pl}
Each element in $X^\prime\,$ is a linear combination of the $\omega_{\xi,\xi}\,$ where $\,\xi\in H.$
\end{proposition}

\begin{proof} By the universal representation (\cite{B}~III.5.2.1), positive linear forms from $X^\prime$ are precisely the $\omega_{\xi,\xi}$ for some $\xi\in H.$ On the other hand, each element in $X^\prime$ is represented as a canonical linear combinations of some positive ones (\cite{B}~II.6.3.4).
\end{proof}

%One may use Proposition 2.1 in Takesaki, if it is correct.

\bigskip

Given $\phi\in X^\prime,\,$ we regard $\phi\in \mathcal{B}(H)^\prime\,$ via the above canonical representation. Also, from Proposition~\ref{pl}, the following is immediate.

\bigskip

\begin{lemma}\label{lpe}
Let $a,b\in{\rm Fin}(\s X).\,$ Then $a\approx b\,$ iff $\s\omega_{\xi,\rho}(a)\,=\,\s\omega_{\xi,\rho}(b)\,$ for all $\xi,\rho\in H.\,$
\end{lemma}

\begin{proof} Since $\,\omega_{\xi,\rho}\in X^\prime\,$ for $\,\xi,\rho\in H,\,$ one direction is trivial.

For the other one, assume that $\s\omega_{\xi,\rho}(a)\,=\,\s\omega_{\xi,\rho}(b)\,$ for all $\xi,\rho\in H.\,$ Then by Proposition~\ref{pl}, $\,\forall \phi\in X^\prime\, \s\phi (a)\,=\, \s\phi ( b),\,$ i.e. $\,a\,\approx\, b.$
\end{proof}

\bigskip

For $a\in{\rm Fin}(\s X)$ and $\xi,\rho\in H,\,$
\begin{equation}\label{omega}
\overline{\s\omega_{\rho,\xi}(a)}\,=\,\overline{\langle a (\s\rho),\s\xi\rangle}\,=\,\langle \s\xi,a (\s\rho)\rangle\,=\,\langle a^* (\s\xi),\s\rho\rangle\,=\,\,\s\omega_{\xi,\rho}(a^*).
\end{equation}
Therefore we have:

\bigskip

\begin{corollary}
Let $a,b\in{\rm Fin}(\s X).\,$ Then $a\approx b\,$ iff $a^*\approx b^*.$\hfill $\Box$
\end{corollary}

\bigskip

For fixed $a\in{\rm Fin}(\s X)$ and $\rho\in H,$ by (\ref{omega}), the mapping $H\ni \xi \mapsto \overline{\s\omega_{\rho,\xi}(a)}$ is a bounded linear form on $\,H.\,$ Hence, by the Riesz-Fr\`{e}chet Theorem, there is a unique $\eta\in H$ such that $\forall \xi\in H\; \langle \xi, \eta \rangle = \st\overline{\s\omega_{\rho,\xi}(a)}.$

Moreover, for $\xi,\rho\in H,$ if $b\in{\rm Fin}(\s X)$ and $a\approx b,$ then $\overline{\s\omega_{\rho,\xi}(a)}\,\approx\,\overline{\s\omega_{\rho,\xi}(b)},\,$ since  $\omega_{\rho,\xi }(x)\in X^\prime.$ We thus define
\[\pi: X^{\prime\prime} \to \mathcal{B}(H)\]
by letting $\pi (\widehat{a})$ be the bounded operator on $H$ that takes $\rho\in H$ to this unique $\eta\in H.\,$ That is, for $\rho\in H,\,$ $\pi(\widehat{a})(\rho )\,$ is the unique element in $H$ such that
\[\forall\xi\in H\;\big[  \langle \xi, \pi(\widehat{a})(\rho )\rangle \,=\,\st \langle \s\xi, a(\s\rho)\rangle\big].\]

\bigskip

In case the preimage or the image of $\pi$ is in $X,$ we have the following.

\bigskip

\begin{proposition}\label{ppg}
Let $a\in{\rm Fin}(\s X).\,$
\begin{enumerate}
  \item [(i)] If $a\in X,$ then $\pi (a) =a.$
  \item [(ii)] If $\pi (\widehat{a})\in X\,$ then $\dis \widehat{a}\,=\,\widehat{\s (\pi (\widehat{a}))}.\,$ In particular $a\approx c\,$ for some $c\in X.$
\end{enumerate}
 \end{proposition}

\begin{proof}
(i): If $a\in X,\,$ then $\,\forall\,\xi,\rho\in H,\,$ we have
\[\langle\xi,\pi(\widehat{a})(\rho) \rangle\,\approx\,\langle\s\xi,a(\s\rho) \rangle\,=\,\langle\xi,a(\rho) \rangle. \]

(ii): Let $\,b=\pi (\widehat{a})\in X.\,$ For $\,\xi,\rho\in H,\,$ we have
\[\overline{\s\omega_{\rho,\xi}(a)}\,\approx\,\langle\xi,\pi(\widehat{a})(\rho) \rangle\,=\,\langle\xi,b(\rho) \rangle\,\approx\,\langle\s\xi, \s b(\s \rho) \rangle\,=\,\overline{\s\omega_{\rho,\xi}(\s b)}. \]
So we have the required $\,a\approx \s b$ as a consequence of Lemma~\ref{lpe}. Note $b\in X.$
\end{proof}

\bigskip

The following shows that $\pi$ is an isometry.

\bigskip

\begin{lemma}\label{le}
Let $a\in{\rm Fin}(\s X).$ Then $\norm{\widehat{a}}_{X^{\prime\prime}}\,=\,\norm{\pi(\widehat{a})}_{\mathcal{B}(H)}.$
\end{lemma}

\begin{proof}
First note that
\begin{align}\label{le1}
\norm{\pi(\widehat{a})}_{\mathcal{B}(H)}\,&=\,\sup\big\{\langle\xi,\pi(\widehat{a})(\rho) \rangle\,\big\vert\, \xi,\rho\in H\,\land\, \norm{\xi}=\norm{\rho}=1\,\big\}\\
\nonumber &=\,\sup\big\{\st\langle\s\xi,a (\s\rho) \rangle\,\big\vert\, \xi,\rho\in H\,\land\, \norm{\xi}=\norm{\rho}=1\,\big\}\\
\nonumber &=\,\sup\big\{\st\abs{\s\omega_{\rho,\xi}(a)}\,\big\vert\, \xi,\rho\in H\,\land\, \norm{\xi}=\norm{\rho}=1\,\big\}\,\leq \,\norm{\widehat{a}}_{X^{\prime\prime}}.
\end{align}

We define for an internal subspace $K\subset \s H$ the following
\[\Theta (K)\,=\, \sup\big\{\abs{\s\omega_{\xi,\rho}(a)}\,\big\vert\, \xi,\rho\in K\,\land\, \norm{\xi}=\norm{\rho}=1\,\big\}.\]

Let $\,r\,=\,\norm{\pi(\widehat{a})}_{\mathcal{B}(H)}.$ Then by (\ref{le1}) and the transfer, the family $\dis\{\mathcal{F}_{H_0, n}\,\},\,$ where
\[\mathcal{F}_{H_0, n}\,:=\,\big\{K\,\vert\, K\; \text{an internal subspace so that}\; H_0\subset K\subset \s H\;\land\;\abs{\Theta (K)-r}\leq\frac{1}{n}\,\big\},\,\] with indices ranging over all finite dimensional subspace $H_0\subset H\,$ and $n\in\N,$ has the finite intersection property. Therefore, as a consequence of enough saturation, we can fix some (hyperfinite dimensional) internal subspace $K\subset \s H\,$ such that $\,H\subset K\,$ and $\Theta (K)\,\approx\,r.$

Let $\varrho_K$ denote the projection of $\s H$ onto this $K.$ Then for all $\xi,\rho\in H\,$ we have $\s\omega_{\xi,\rho}(a)\,=\,\s\omega_{\xi,\rho}(a\,\varrho_K ).\,$ (If $\varrho_K\in \s X,$ we have $\,a\,\approx\, a\,\varrho_K\,$ by Lemma \ref{lpe}.)

Now notice that
\begin{align*}
\norm{\widehat{a}}_{X^{\prime\prime}}\,&=\,\sup\big\{\st\abs{\s\phi (a)}\,\big\vert\, \phi\in X^\prime,\; \norm{\phi}=1  \big\}\\
                  &=\,\sup\big\{\st\abs{\s\phi (a\,\varrho_K )}\,\big\vert\, \phi\in X^\prime,\; \norm{\phi}=1  \big\}\,\lessapprox\,\norm{a\,\varrho_K}_{\s \mathcal{B}(H)},
\end{align*}
where in the second equality, the $\phi$ is regarded as an element in $\mathcal{B}(H)^\prime\,$ via the canonical linear combination of positive linear functionals.

Observing that $\norm{a\,\varrho_K}_{\s \mathcal{B}(H)}\,=\,\Theta (K),\,$ we therefore have
\begin{equation}\label{le2}
\norm{\widehat{a}}_{X^{\prime\prime}}\,\lessapprox\,\norm{a\,\varrho_K}_{\s \mathcal{B}(H)}\,=\,\Theta (K)\,\approx\, r\,=\,\norm{\pi(\widehat{a})}_{\mathcal{B}(H)},
\end{equation}

Now the conclusion follows from (\ref{le1}) and (\ref{le2}).
\end{proof}

\bigskip

\begin{lemma}\label{lpf}
The embedding $\pi:X^{\prime\prime}\to\mathcal{B}(H)\,$ is a Banach space isometric isomorphism.
\end{lemma}

\begin{proof}
For $a,b\in{\rm Fin}(\s X)\,$ and $\lambda\in\C,\,$ if $\,\xi,\rho\in H,\,$
\begin{align*}
\langle\xi,\pi(\widehat{a}+\lambda\widehat{b})(\rho)\rangle\,&\approx\,\s\omega_{\xi,\rho}((a+\lambda b)^*)\,=\,\s\omega_{\xi,\rho}(a^*+\bar{\lambda} b^*)\\
   & \,=\,\s\omega_{\xi,\rho}(a^*)+\bar{\lambda} \s\omega_{\xi,\rho}(b^*))\\
   & \,=\,\langle\xi,\pi(\widehat{a})(\rho)\rangle+\bar{\lambda}\langle\xi,\pi(\widehat{b})(\rho)\rangle
     \,=\,\langle\xi,\,\pi(\widehat{a})(\rho)+\lambda\pi(\widehat{b})(\rho)\rangle,
\end{align*}
hence $\pi$ is a linear operator.

It now follows from Lemma~\ref{le} that $\pi$ is an isometric isomorphism.
\end{proof}

\bigskip

In the remainder of this section, we will show that $\pi$ is in fact an isometric C*-isomorphism.

Given $\, \widehat{a}, \widehat{b} \in X^{\prime\prime},\,$ the mapping that takes $\,\omega_{\xi,\rho}\,$ (restricted on $X$),  where $\,\xi,\rho\in H,\,$ to $\,\langle \xi,\, \pi(\widehat{a})\big(\pi (\widehat{b}) (\rho )\big)\rangle$ extends to a bounded linear form in $\,X^\prime\,$ via Proposition~\ref{pl}. So this bounded linear form is the unique $\widehat{c^* }$ for some $\, c \in{\rm Fin}(\s X).\,$ Then for $\,\xi,\rho\in H,\,$
\[\langle \xi,\, \pi(\widehat{a})\big(\pi (\widehat{b}) (\rho )\big)\rangle\,=\,\widehat{c^*} (\omega_{\xi,\rho})\,\approx\,\s\omega_{\xi,\rho}(c^*)\,=\, \langle \s\xi,\, c(\s\rho)\rangle\,\approx\, \langle \xi,\, \pi(\widehat{c})(\rho)\rangle,\]
where the second equality follows from (\ref{omega}).

In other words, we have $\pi(\widehat{c})\,=\,\pi(\widehat{a})\pi(\widehat{b}).\,$ Hence the image of $\pi$ is closed under the product in $\mathcal{B}(H).\,$

Also for $\,\xi,\rho\in H,\,$
\begin{align*}
\langle \xi,\, (\pi(\widehat{a}))^* (\rho)\rangle\,&=\,\langle \pi(\widehat{a})(\xi),\,\rho\rangle\,=\,\overline{\langle \rho,\,\pi(\widehat{a})(\xi)\rangle}\,\approx\, \overline{\langle \s\rho,\,a(\s\xi)\rangle}\\
     &\approx\, \langle a(\s\xi),\,\s\rho \rangle\,=\,\langle \s\rho,\, a^* (\s\xi) \rangle.
\end{align*}
i.e. $\,(\pi (\widehat{z})^*\,=\,\pi (\widehat{a^*})\,$ and the image of $\pi$ is closed under the involution in $\mathcal{B}(H).\,$

These together with Lemma~\ref{lpf} prove that:

\bigskip

\begin{theorem}\label{tcs}
$\pi$ is an isometric C*-isomorphism embedding $\,X^{\prime\prime}\,$ into $\mathcal{B}(H).\,$ \hfill $\Box$
\end{theorem}

\bigskip

Recall from \cite{B} that the weak operator topology is the weakest topology on $\mathcal{B}(H)$ that makes all $\omega_{\xi,\rho},\,\xi,\rho\in H,\,$ continuous. It is therefore generated by open sets of the form $\,\{x\in \mathcal{B}(H)\,\vert\, \langle \xi,\,x(\rho) \rangle<\epsilon\,\},\;\xi,\rho\in H,\, \epsilon>0.\,$ A von Neumann algebra is a C*-algebra which is closed in the weak operator topology in some representation as a C*-subalgebra of the algebra of bounded operators on some Hilbert space.

\bigskip

\begin{corollary}
{\rm (Sherman-Takeda Theorem)} The von Neumann algebra generated by a C*-algebra $X\,$ is isometrically C*-isomorphic to  $X^{\prime\prime}.$
\end{corollary}

\begin{proof}
By Proposition~\ref{ppg} (i) and Theorem~\ref{tcs}, $\pi$ extends $X$ isometrically and isomorphically to a C*-subalgebra of $\mathcal{B}(H).\,$ Moreover, the image is the weak operator topological closure of $X:$

Let $\widehat{a}\in X^{\prime\prime}\,$ and consider $\pi(\widehat{a}).\,$ Suppose $\,\xi_i,\rho_i\in H,\, \epsilon_i>0,\; i=1,\dots, n\in\N,$
\[\bigwedge_{i=1}^n\,\abs{\langle \xi_i,\,\pi(\widehat{a})(\rho_i) \rangle}<\epsilon_i.\]
Then for some $\epsilon^{\prime}_i<\epsilon_i,$ we have
\[\bigwedge_{i=1}^n\,\abs{\langle \s\xi_i,\,a(\s\rho_i) \rangle}\leq\epsilon^{\prime}_i,\quad\text{consequently}\quad \exists x\in\s X\bigwedge_{i=1}^n\,\abs{\langle \s\xi_i,\,x(\s\rho_i) \rangle}\leq\epsilon^{\prime}_i.\]
By transfer, we then have
\[\exists x\in X\bigwedge_{i=1}^n\,\abs{\langle \xi_i,\,x(\rho_i) \rangle}\leq\epsilon^{\prime}_i.\]
So $\pi(\widehat{a})$ is in the weak operator topological closure of $X.\,$

Hence the image of $\pi$ is the von Neumann algebra generated by $X$ in $\mathcal{B}(H).$
\end{proof}

\bigskip

\begin{remark}\begin{enumerate}
  \item [(i)] The above shows that the weak nonstandard hull of a C*-algebra is always a von Neumann algebra.
  \item [(ii)] The predual of the von Neumann algebra generated by $X$ is simply $X^\prime.$ (See also Sakai's Theorem (\cite{B}~III.2.4.2).)
  \item [(iii)] The bicommutant of $X$ in $\mathcal{B}(H)$ is just $X^{\prime\prime}.$ \hfill $\Box$
\end{enumerate}
\end{remark}

\bigskip

By results in \cite{G}, the product on $X^{\prime\prime}$ given by $\widehat{a}\,\widehat{b} := \pi^{-1}\big(\pi(\widehat{a})\pi (\widehat{b})\big)$ is the same as the left and the right Arens product. But in the current setting there is a better representation of the Arens product deriving from the convergent nets in the weak operator topology.

\bigskip

\begin{theorem}\label{arens}
Given $a,b\in{\rm Fin}(\s X),\,$ there is $a_0\in{\rm Fin}(\s X),\,a_0\approx a,\,$ such that $\pi(\widehat{a})\pi(\widehat{b})\,=\,\pi(\widehat{a_0 b}).\,$

Similarly, there is $b_0\in{\rm Fin}(\s X),\,b_0\approx b,\,$ such that $\pi(\widehat{a})\pi(\widehat{b})\,=\,\pi(\widehat{ab_0}).\,$

Moreover, $a_0$ (resp. $b_0\,$) can be chosen from the internal extension of a net converging to $\pi(\widehat{a})$ (resp. $\pi(\widehat{b})\,$) in the weak operator topology.
\end{theorem}

\begin{proof}
Let $a,b\in{\rm Fin}(\s X),\,$ write $c=a^*.\,$ Let $c_i\in X,\,i\in I,$ a net, and $c_i\to \pi(\widehat{c})\,$ in the weak operator topology. Then for any finite list of $\xi,\,\rho\in H,\,n\in \N,\,$ we have
\[\abs{\langle c_i(\s\xi),\,\s \big(\pi (\widehat{b})(\rho )\big)\rangle\,-\,\langle c(\s\xi),\,\s \big(\pi (\widehat{b})(\rho )\big)\rangle}\leq\,n^{-1}\quad\text{for all large}\; i\in I.\]
Note for $i\in I\,$ that
\[\langle c_i(\s\xi),\,\s \big(\pi (\widehat{b})(\rho )\big)\rangle\,=\,\langle c_i(\xi),\,\big(\pi (\widehat{b})(\rho )\big)\rangle\,\approx\,
\langle \s (c_i(\xi)),\,b(\s\rho )\big)\rangle\]
Therefore the family $\dis\{\mathcal{F}_{H_0, n, i}\,\}\,$ has the finite intersection property, where the indices range over all finite dimensional subspace $H_0\subset H,\,$ $n\in \N\,$ and $i\in I,\,$ and $\mathcal{F}_{H_0, n, i}$ is the internal set of the $(K,k),$ where $K$ is an internal subspace of $\s H$ and $K$ includes $H_0$ as a subspace, $i\leq k\in \s I,\,$ for all $\, \xi,\rho\in K\,$ with $\norm{\xi}=\norm{\rho}=1,$ the following is satisfied:
\[\abs{\langle c_k(\s\xi),\,b(\s \rho )\big)\rangle\,-\,\langle c(\s\xi),\,\s \big(\pi (\widehat{b})(\rho )\big)\rangle}\leq\,n^{-1}.\]
By enough saturation, we let $(K,k)\,$ be an element in the intersection of the family $\dis\{\mathcal{F}_{H_0, n, i}\,\}.\,$ Then for $\xi,\,\rho\in H\,$ (hence $\s\xi,\,\s\rho\in K\,$), we have
\[\langle c_k(\s\xi),\,b(\s \rho )\big)\rangle\,\approx\, \langle c(\s\xi),\,\s \big(\pi (\widehat{b})(\rho )\big)\rangle\,=\,
\langle \s\xi,\,a\big(\s \big(\pi (\widehat{b})(\rho )\big)\big)\rangle\,\approx\, \langle \xi,\,\pi(a)\big(\pi(\widehat{b})(\rho)\big)\rangle. \]

On the other hand,
\[\langle c_k(\s\xi),\,b(\s \rho )\big)\rangle\,=\,\langle \s\xi,\, c_k^* \big(b(\s \rho )\big)\rangle\,=\, \langle \xi,\, \pi (c_k b)(\s \rho )\rangle.\]
Since $c_k\approx c\,= a^*,\,$ if we let $a_0=c_k^*,\,$ then $a_0\approx a\,$ and the above shows that $\pi(\widehat{a})\pi(\widehat{b})\,=\,\pi(\widehat{a_0 b}).\,$ Note that $a_0\in{\rm Fin}(\s X).$

The second statement of the theorem follows a dual argument, i.e apply the above to $b^* a^*.$
\end{proof}

\bigskip

\begin{remark}
In general, for $a,b\in{\rm Fin}(\s X)\,$  the product $\widehat{a}\widehat{b}$ in $X^{\prime\prime}$ is not the same as $\widehat{ab}.\,$ For example, take $H=\ell^2(\N)$ and $X=\mathcal{B}(H).$  Fix $N\in\s\N\setminus\N.\,$ Let $a\in\s X$ be the operation that interchanges the first and the $N$th coordinates in each $\xi\in H.\,$ Then $a^2=1\,$ so $\widehat{a^2}=1\,$ but $\widehat{a}\,$ is the operator that replaces the first coordinate in each $\xi\in H$ by $0.$\hfill $\Box$
\end{remark}

\vfill

\bibliographystyle{amsplain}

\vfill

\end{document}